\documentclass[12pt]{article}
\usepackage{amsmath,amsthm,amssymb,enumerate,mathscinet}
\usepackage{fullpage}
\usepackage{graphicx}

\newtheorem{theorem}{Theorem}[section]

\newtheorem{lemma}[theorem]{Lemma}

\newtheorem{thm}[theorem]{Theorem}
\newtheorem{prop}[theorem]{Proposition}

\newtheorem{conj}[theorem]{Conjecture}

\newtheorem{proble}[theorem]{Problem}

\numberwithin{equation}{section}
\def\A{\mathcal{A}}
\def\B{\mathcal{B}}
\def\C{\mathcal{C}}
\def\F{\mathcal{F}}
\def\G{\mathcal{G}}
\def\HH{\mathcal{H}}
\def\P{\mathcal{P}}
\def\Q{\mathcal{Q}}
\def\S{\mathcal{S}}
\def\T{\mathcal{T}}
\def\W{\mathcal{W}}
\def\X{\mathcal{X}}
\def\w{\omega}

\def\eps{\varepsilon}

\def\xyz{\frac{s}{kn}}

\def\COMMENT#1{}
\let\COMMENT=\footnote

\title{Kleitman's conjecture about families of given size minimizing the number of $k$-chains} 

\author{J\'ozsef Balogh\footnote{Department of Mathematical Sciences,
 University of Illinois at Urbana-Champaign, Urbana, Illinois 61801, USA, {\tt
jobal@math.uiuc.edu}. Research is partially supported by NSF Grant DMS-1500121, Arnold O. Beckman Research Award (UIUC Campus Research Board 15006).
}
 ~and Adam Zsolt Wagner\footnote{University of Illinois at Urbana-Champaign, Urbana, Illinois 61801, USA, {\tt
zawagne2@illinois.edu}. }}

\begin{document}
\maketitle
\begin{abstract}
A central theorem in combinatorics is Sperner's Theorem, which determines the maximum size of a family $\F\subseteq \P(n)$ that does not contain a $2$-chain $F_1\subsetneq F_2$. Erd\H{o}s later extended this result and determined the largest family not containing a $k$-chain $F_1\subsetneq \ldots \subsetneq F_k$. Erd\H{o}s and Katona and later Kleitman asked how many such chains must appear in families whose size is larger than the corresponding extremal result.

This question was resolved for $2$-chains by Kleitman in $1966$, who showed that amongst families of size $M$ in $\P(n)$, the number of $2$-chains is minimized by a family whose sets are taken as close to the middle layer as possible. He also conjectured that the same conclusion should hold for all $k$, not just $2$. The best result on this question is due to Das, Gan and Sudakov who showed that Kleitman's conjecture holds for families whose size is at most the size of the $k+1$ middle layers of $\P(n)$, provided $k\leq n-6$. Our main result is that for every fixed $k$ and $\eps>0$, if $n$ is sufficiently large then Kleitman's conjecture holds for families of size at most $(1-\eps)2^n$, thereby establishing Kleitman's conjecture asymptotically. Our proof is based on ideas of Kleitman and Das, Gan and Sudakov. Several open problems are also given.
\end{abstract}

\section{Introduction}
Denote by $\Sigma(n,r)$ the size of the $r$ largest layers in $\P(n)$, that is, $\Sigma(n,r)=\sum_{i=\lceil \frac{n-r+1}{2}\rceil}^{\lceil \frac{n+r-1}{2}\rceil}\binom{n}{i}$. Sperner's Theorem \cite{sperner}, a cornerstone result in extremal combinatorics from $1928$, states that the size of the largest family $\F\subseteq \P(n)$ that does not contain two sets $F_1,F_2\in\F$ with $F_1\subsetneq F_2$ is $\binom{n}{\lfloor n/2 \rfloor}$. This result was extended by Erd\H{o}s~\cite{erdos}, who showed that the size of the largest family without a $k$-chain, that is, $k$ sets $F_1\subsetneq \ldots \subsetneq F_k$, is the sum of the $k-1$ largest binomial coefficients, $\Sigma(n,k-1)$.

The following natural question was first posed by Erd\H{o}s and Katona and then extended by Kleitman some fifty years ago. Given a family $\F$ of $s$ subsets of $[n]$, how many $k$-chains must $\F$ contain? We denote this minimum by $c_k(n, s)$, and determine it for a wide range of values of $k$ and $s$. For $k=2$ this question was completely resolved by Kleitman~\cite{kleitmanpairs}. We say that a family $\F\subseteq \P(n)$ is \emph{centered} if  for any two sets $A,B\subseteq [n]$ with $A\in \F$ and $B\notin \F$ we have that $|n/2-|A||\leq|n/2-|B||$, and if $|n/2-|A||=|n/2-|B||$ then we have $|A|\geq |B|$. That is, if $\F$ is constructed by ``taking sets whose size is as close to $n/2$ as possible'' (and if two layers have the same size we fill up the top one first).  Equipped with this definition Kleitman's theorem is as follows.
\begin{theorem}[\label{kleitman}Kleitman~\cite{kleitmanpairs}]
Let $n,M>0$ be integers. Amongst families $\F\subseteq \P(n)$ of size $M$, the number of $2$-chains in $\F$ is minimized by a centered family.
\end{theorem}
Note that Theorem~\ref{kleitman} does not claim that centered families are the \emph{only} families achieving this minimum, which is not quite true (but close!). The families achieving minimum in Theorem~\ref{kleitman} have been completely characterized by Das--Gan--Sudakov \cite{dasgansuda}.

In the present paper we are interested in what happens for $k>2$. Kleitman conjectured that the conclusion of Theorem~\ref{kleitman} should hold for $k>2$ as well, that is, the number of $k$-chains in $\F$ is minimized if $\F$ is obtained by taking sets whose size is as close to $n/2$ as possible.

\begin{conj}[Kleitman, \cite{kleitmanconj1,kleitmanpairs}]\label{kleitmanconjecture}
Let $n,M>0$ and $k\geq 2$ be integers. Amongst families $\F\subseteq \P(n)$ of size $M$, the number of $k$-chains in $\F$ is minimized by a centered family.
\end{conj}

Similarly as before there may be other families minimizing the number of $k$-chains that are not centered, and Kleitman's conjecture does not say anything about them. In reality it is probably true that all minimizing families are very close to being centered, and in particular have at most two partially filled layers. We are  far from proving any such statement. 

Only little progress has been made towards Kleitman's conjecture so far. Dove--Griggs--Kang--Sereni~\cite{griggs} and independently Das--Gan--Sudakov~\cite{dasgansuda} proved that Kleitman's conjecture is true for families whose size is at most the size of the $k$ middle layers. For a family $\F\subseteq \P(n)$, write $c_k(\F)$ for the number of $k$-chains contained in $\F$.
\begin{thm}[Das--Gan--Sudakov~\cite{dasgansuda}, Dove--Griggs--Kang--Sereni~\cite{griggs}]
Let $k,M,n>0$ with $M\leq \Sigma(n,k)$. Amongst families $\F\subseteq \P(n)$ of size $M$, the function $c_k(\F)$ is minimized by centered families.
\end{thm}
The first set of authors obtained stability versions of the above theorem as well. The current best result on Kleitman's conjecture is due to Das--Gan-Sudakov~\cite{dasgansuda}, who showed that Kleitman's conjecture holds for family sizes at most the middle $k+1$ layers, provided $k\leq n-6$.
\begin{thm}[\label{dasgansudakovthm}Das--Gan--Sudakov]
Let $n\geq 15$, $M\leq \Sigma(n,k+1)$ and $k\leq n-6$. Amongst families $\F\subseteq \P(n)$ of size $M$, the function $c_k(\F)$ is minimized by centered families.
\end{thm}
Once again they actually obtained slightly stronger results, providing stability results for families for which $c_k(\F)$ is close to the minimum. For more on the history and motivation of this problem we refer the reader to the very well-written introduction of \cite{dasgansuda}. 

Our main result can be viewed as an asymptotic solution to Kleitman's conjecture.
\begin{thm}\label{mainthm}
For every $k$ and $\eps>0$ there exists an $n_0=n_0(k,\eps)$ such that if $n\geq n_0$ and $M\leq (1-\eps)2^n$ then amongst families $\F\subset \P(n)$ of size $M$,  the function $c_k(\F)$ is minimized by centered families.
\end{thm}
Our proof consists of two main parts. First we show that amongst families that are contained in the middle roughly $\sqrt{n\log n}$ layers, centered families are the best (i.e.~they have the smallest $c_k(\F)$). This part is based on the symmetric chain decomposition construction of de Bruijn--Tengbergen--Kruyswijk~\cite{scd} and ideas of Kleitman~\cite{kleitmanpermu} and contains most of the new ideas of the paper. The second part of the proof is then showing that an optimal family cannot contain sets that are too small or too large. Our method of proving this is mostly based on ideas of Das--Gan--Sudakov~\cite{dasgansuda}. Throughout the paper we make no effort to optimize the value of $n_0(k,\eps)$. For the corresponding maximization question, i.e.~determining the maximum possible number of comparable pairs amongst families of size $M$ in $\P(n)$ we refer the reader to~\cite{alonmaxim}.

\section{Set-up}\label{setupsection}

Our main goal of this paper is to prove Theorem~\ref{mainthm}. Hence throughout the paper we consider $k$ and $\eps>0$ to be fixed. We set $n_0$ to be sufficiently large so that all following inequalities hold and want to show that for any $n\geq n_0$ the conclusion of Theorem~\ref{mainthm} holds. For that we fix an arbitrary $M\leq (1-\eps)2^n$.  By Theorem~\ref{dasgansudakovthm} we know that the conclusion of Theorem~\ref{mainthm} holds if $M\leq \Sigma(n,k+1)$ hence we will always assume 
\begin{equation}\label{rkplus1}
M>\Sigma(n,k+1),
\end{equation} recalling that $\Sigma(n,s)$ is the total size of the $s$ biggest layers in $\P(n)$. Let $r$ be defined as the unique integer such that $$\Sigma(n,r-1)<M\leq \Sigma(n,r).$$
That is, we wish to show that one of the optimal families will fully contain the $r-1$ middle layers, and some elements from a neighboring layer. We observe that since for $n$ large enough $$\Sigma(n,\sqrt{n}\log^{1/20}n)>\left(1-\frac{\eps}{2}\right)2^n,$$ we have 
\begin{equation}\label{rupperbound}
r\leq \sqrt{n}\log^{1/10}n.
\end{equation}
Moreover we will assume that 
\begin{equation}\label{kupperbound}
k\leq \log^{1/100}n.
\end{equation}

Throughout the paper most propositions will aim to show that given certain conditions, centered families minimize the number of $k$-chains. Note that every centered family $\F\subset \P(n)$ of size $M$ contains the same number of $k$-chains. It will be convenient for us to pick for each positive integer $Q$ one specific centered family of size $Q$, that we will call $\G_Q$, and show that $\G_M$ minimizes the number of $k$-chains. 
Note that if $\F$ is centered then there exists at most one $j\in [n]$ such that $\emptyset\neq \F\cap\binom{[n]}{j}\neq \binom{[n]}{j}$, and we call this $j$ the \emph{partial layer} of $\F$ if it exists. Moreover if $Q>0$ is a fixed integer then every centered family $\F$ of size $Q$ in $\P(n)$ has the same partial layer $j$ and the same intersection sizes with all layers. Given $Q,n$ the only free choice one has when specifying a centered family of size $Q$ in $\P(n)$ is what to do on the partial layer. A natural choice for $\G_Q$ is to choose an initial segment of the partial layer according to some total order on the elements of $\P(n)$. What ordering we pick makes absolutely no difference in the proof - but we believe it could be helpful for the reader to pick a specific total order. The \emph{lexicographic order} $<_{lex}$ on $\P(n)$ is defined as follows. If $|A|<|B|$ then $A<_{lex} B$. Otherwise if $|A|=|B|$ then if the smallest element of $A\Delta B$ is in $A$ then $A<_{lex} B$, otherwise $B\leq_{lex} A$. 
For any positive integer $Q$ let $\G_Q$ be the centered family of size $Q$ in $\P(n)$ whose intersection with its partial layer $j$ is an initial segment of the lexicographic ordering of $\binom{[n]}{j}$.

We will need to deal with families which are contained in a subset of $\P(n)$, for these it will be useful to extend the above definitions in a natural way. Given a family $\P'\subseteq \P(n)$, say that a family $\F$ \emph{is centered in} $\P'$ if for any two sets $A,B\in \P'$ with $A\in \F$ and $B\notin \F$ we have that $|n/2-|A||\leq|n/2-|B||$, and if $|n/2-|A||=|n/2-|B||$ then we have $|A|\geq |B|$. That is, if $\F$ is constructed by ``taking sets whose size is as close to $n/2$ as possible in $\P'$'' (and if two layers have the same size we fill up the top one first). For a positive integer $Q$ define $\G_{\P',Q}$ to be the family of size $Q$ which is centered in $\P'$ and whose intersection with its partial layer is an initial segment of the restriction of $<_{lex}$ to $\P'$. So the family $\G_Q$ defined above equals $\G_{\P(n),Q}$.

A family $\A=\{A_1,\ldots,A_{\ell}\}\subset \P(n)$ is a \emph{chain} if $A_1\subsetneq \ldots \subsetneq A_{\ell}$. We say $\A$ is a chain with \emph{step sizes} $a_1,\ldots,a_{\ell -1}$ if $|A_{i+1}\setminus A_i|=a_i$ for all $i\in [\ell-1]$.  For a family $\F\subseteq \P(n)$  and integers $a_1,\ldots,a_{k-1}\geq 1$, define 
$$\Phi^*(\F,a_1,\ldots,a_{k-1}):=\{(A_1,\ldots,A_k)\in\F^k: A_1\subsetneq \ldots \subsetneq A_k, \text{ and } |A_{i+1}\setminus A_i|= a_i \text{ for all } i\in [k-1]\},$$ the set of $k$-chains with precisely these step sizes in $\F$. Given a $k$-chain $\A=\{A_1,\ldots,A_k\}$ with $A_1\subsetneq\ldots\subsetneq A_k$, define
$$d(\A):=\max\{||A_k|-n/2|, ||A_1|-n/2|\}.$$

For every fixed $\mathbf{a}=(a_1,\ldots,a_{k-1})$ we fix a total order $<^*_{\F,\mathbf{a}}$ on $\Phi^*(\F,\mathbf{a})$ that satisfies the following property:

\vspace{.2in}

\emph{For every positive integer $Q$ the family $\Phi^*(\G_{\F,Q},\mathbf{a})$ is an initial segment of the order $<^*_{\F,\mathbf{a}}$.}

\vspace{.2in}

Note that such an ordering $<^*_{\F,\mathbf{a}}$ exists because $\G_{\F,Q}\subsetneq \G_{\F,Q+1}$ for all $0\leq Q\leq |\F|-1$.

\textbf{Notation.} Wherever possible we use standard notation and for the variable names we aim to follow the notation of~\cite{dasgansuda}. There are two notational oddities that we feel we should mention. Firstly, for chains we use cursive capital letters, e.g. $\A,\B$, etc. - however, later in the paper we will deal with hypergraphs on vertex set $\P(n)$ with edges corresponding to some chains, whence we will refer to the edges as $e,f$, etc. Several times we will, without mentioning this explicitly, make use of the natural correspondence between such edges and chains and hence occasionally label chains as $e,f$, etc. wherever this does not create confusion. Secondly, since we often consider the step sizes  $a_1,\ldots,a_{k-1}$ of a chain, for sake of brevity and cleanliness we will sometimes abbreviate this list as $\mathbf{a}$, with the understanding that $\mathbf{a}=(a_1,\ldots,a_{k-1})$. We will always assume without mentioning it explicitly that the variable $\mathbf{a}$ refers to a list $(a_1,\ldots,a_{k-1})$ of positive integers corresponding to the step sizes of some chain. Moreover, whenever the variable $\mathbf{a}$ and the list $a_1,\ldots,a_{k-1}$ or $\{a_i\}_{i=1}^{k-1}$ are used in the same context they will refer to the same thing.

\section{Families close to being centered}

Set $$d=\lfloor 10k\sqrt{n\log n}\rfloor$$ and write $\P_{n,d}$ for the union of the $d$ middle layers in $\P(n)$, that is, for the family of sets $A\in \P(n)$ with $\lceil\frac{n-d+1}{2}\rceil\leq |A|\leq \lceil\frac{n+d-1}{2}\rceil$. Recall that we fixed an $M$ at the very beginning of Section~\ref{setupsection}, which denotes the size of the families we will ultimately be interested in. Our goal in this section is to show the following proposition:
\begin{prop}\label{mainmiddle}
Amongst all families $\F\subseteq \P_{n,d}$ of size $M$, the number of $k$-chains in $\F$ is minimized if $\F=\G_M$. 
\end{prop}
Once again we do not claim that $\G_M$ is the only family minimizing the number of $k$-chains. Once we have shown this proposition the only remaining step is to show that an optimal family cannot contain sets that are very far from the middle layer. This will be done later, in Section~\ref{smalllargesetssection}.

The proof of Proposition~\ref{mainmiddle} uses the standard technique of compressions. Given a suboptimal family we show that we can apply some operations to it to make it better (in a sense defined later). One of the main ideas of the proof is that instead of moving the sets in the family (as in standard compression techniques), we view the family as a collection of chains and apply compression to the chains instead of directly to the family. One interesting aspect of this compression is that if we apply it to a family $\F$ we get an object that does not usually correspond to a family $\F'\subset \P(n)$ - instead the object we obtain after compressing a family will be a subgraph (equipped with a measure) of a weighted hypergraph, whose edges correspond to chains in $\P(n)$. In this hypergraph the induced subhypergraphs correspond to our usual families, but in order to make our compression methods work we have to leave the world of standard families and enter the realm of these more general objects (which we will refer to as \emph{measured subhypergraphs}). Hence in order to prove Proposition~$3.1$ we will in fact show that amongst all such objects that have the same 'size' as our family $\F$, the ones corresponding to centered families cannot be improved by compressions and then deduce Proposition~\ref{mainmiddle} from this.

\subsection{Definitions}
We say $\A=\{A_1,\ldots,A_{\ell}\}$ is a chain with \emph{step sizes} $a_1,\ldots,a_{\ell -1}$ if $|A_{i+1}\setminus A_i|=a_i$ for every $i\in [\ell-1]$. It has step sizes \emph{at least} $a_1,\ldots,a_{\ell-1}$ if $|A_{i+1}\setminus A_i|\geq a_i$ for all $i\in [\ell-1]$. The \emph{height} of the chain $\A$ is defined as $h(\A):=|A_\ell\setminus A_1|$. It is called a \emph{downward} chain if  $||A_{\ell}|-n/2|\geq ||A_1|-n/2|$, otherwise we call it an \emph{upward} chain.
 We call $\A$ a \emph{skipless chain} if it is a chain and $|A_{i+1}\setminus A_i|=1$ for all $i\in[\ell -1]$. Moreover $\A$ is a \emph{symmetric chain} if it is a skipless chain and $n/2-|A_1|=|A_{\ell}|-n/2$. That is, a symmetric chain starts at some level $s$, ends at level $n-s$ and contains precisely one set from each level in between. A \emph{symmetric chain decomposition} (SCD in short) of $\P(n)$ is a partition of $\P(n)$ into disjoint symmetric chains, whose union is the entire $\P(n)$. It is not a priori obvious that an SCD of $\P(n)$ should exist for all $n$ - this was showed by de Bruijn--Tengbergen--Kruyswijk~\cite{scd}.  Note that as every symmetric chain intersects the middle layer in precisely one element, every SCD of $\P(n)$ consists of precisely $N:=\binom{n}{\lfloor n/2 \rfloor}$ chains.

Let $\A\subset \P(n)$ be a chain and $\X=\{X_1,\ldots,X_N\}$ be an SCD of $\P(n)$. We say $\X$ \emph{contains} $\A$ if there exists an $i\in [N]$ such that every set in $\A$ is contained in the chain $X_i$. For a chain $\A\subset \P(n)$ we define its \emph{weight} $\w(\A)$ to be the probability that $\A$ is contained in an $SCD$ that is chosen uniformly at random from the collection of all SCDs of $\P(n)$. This weight can be calculated easily, as shown by the following proposition.

\begin{prop}\label{weightformula}
Let $\ell$ be an arbitrary positive integer and let $\A=\{A_1,\ldots,A_{\ell}\}\subset \P(n)$ be a chain with $A_1\subsetneq \ldots \subsetneq A_{\ell}$. If $\A$ is a downward chain then 
$$\w(\A)=\prod_{i=1}^{\ell-1}\binom{|A_{i+1}|}{|A_i|}^{-1}=\frac{|A_1|!}{|A_{\ell}|!}\prod_{i=1}^{\ell -1}|A_{i+1}\setminus A_i|!.$$
If $\A$ is an upward chain then $$\w(\A)=\prod_{i=1}^{\ell-1}\binom{n-|A_{i}|}{n-|A_{i+1}|}^{-1} = \frac{(n-|A_{\ell}|)!}{(n-|A_1|)!}\prod_{i=1}^{\ell -1}|A_{i+1}\setminus A_i|!.$$
\end{prop}

\subsection{Properties of the weight function}
There are two reasons for why we chose this probability for the weight $\w(\A)$ of a set. The first one is that it will imply that, under suitable conditions, if $\A,\B$ are two chains with $h(\A)<h(\B)$ then we will have $\w(\A)\gg \w(\B)$. The second reason is that it will allow us to formulate a natural and best possible weighted supersaturation statement, essentially showing that centered families minimize the total weight of $k$-chains that they contain. The hard part will be to show that this implies that they also minimize the \emph{number} of $k$-chains.

We start by proving the formulae claimed in the previous subsection.
\begin{proof}[Proof of Proposition~\ref{weightformula}]
Let $\A=\{A_1,\ldots,A_{\ell}\}\subset \P(n)$ be a downward chain with $A_1\subsetneq \ldots \subsetneq A_{\ell}$, the proof of the upward case is identical. Let $\X$ be a SCD chosen uniformly at random from the collection of all SCDs of  $\P(n)$ and let $X$ be the chain in $\X$ that contains $A_{\ell}$.  Since $X$ is a symmetric chain and $\A$ is downward we have that for each $i\in [\ell]$, the chain $X$ contains precisely one element of size $|A_i|$ (and possibly some others). Let $B_i$ be the event that $A_i\in X$. Then
\begin{equation*}
\begin{split}
P(B_1\cap \ldots \cap B_{\ell})&=P(B_{\ell})P(B_1\cap\ldots\cap B_{\ell -1}|B_{\ell})\\
&= P(B_{\ell})P(B_{\ell -1}|B_{\ell})P(B_1\cap\ldots\cap B_{\ell -2}|B_{\ell -1}\cap B_{\ell})\\
&=P(B_{\ell})P(B_{\ell -1}|B_{\ell})P(B_1\cap\ldots\cap B_{\ell -2}|B_{\ell -1})\\
&=\ldots\\
&=P(B_{\ell})\cdot P(B_{\ell -1}|B_{\ell})\cdot P(B_{\ell -2}|B_{\ell -1})\cdot \ldots \cdot P(B_1 | B_2)\\
&=1\cdot\binom{|A_{\ell}|}{|A_{\ell -1}|}^{-1}\cdot \binom{|A_{\ell -1}|}{|A_{\ell -2}|}^{-1}\cdot \ldots\cdot\binom{|A_2|}{|A_1|}^{-1}.
\end{split}
\end{equation*} 
\end{proof}
Note that if $\A,\B$ are two downward $\ell$-chains with $|A_{\ell}|=|B_{\ell}|$ and they have the same step sizes (but possibly in a different order) then they have the same weight.
Let us continue with the next claimed property of the weight function. Let $\A$ and $\B$ be two $\ell$-chains with $h(\A)<h(\B)$, recalling the definition that if $\A$ is an $\ell$-chain then $h(\A)=|A_{\ell}\setminus A_1|$. Note that it is not always the case that $\w(\A)>\w(\B)$ - for instance the chain $\{\emptyset, [n]\}$ has maximal weight ($=1$) and maximal height. But if we avoid wandering too far off from the middle layer then our claim will hold.

\begin{prop}\label{weightsdecreasealot}
Let $a_1,\ldots,a_{k-1}$ and $b_1,\ldots,b_{k-1}$ be positive integers such that $a_i\leq b_i$ for all $i\in [k]$ and strict inequality holds for at least one $i$. Suppose that $\sum_i a_i \leq \sqrt{n}\log^{2/5}n$ and let $\A,\B$ be $k$-chains in $\P_{n,d}$ with step sizes $\{a_i\}_{i=1}^{k-1}$ and $\{b_i\}_{i=1}^{k-1}$ respectively. Then $\w(\A)\geq \w(\B) n^{(h(\B)-h(\A))/3}$.
\end{prop}
\begin{proof}
Without loss of generality we may assume that both chains $\A,\B$ are downward, the proof is similar if one (or both) of them is upward. Then by Proposition~\ref{weightformula} we have 
$$\w(\A)=\frac{|A_1|!\prod_i a_i!}{|A_k|!}, \qquad \w(\B)=\frac{|B_1|! \prod_i b_i!}{|B_k|!}.$$
Then we get (using the falling factorial notation $s_{(t)}=s(s-1)\cdot\ldots\cdot (s-t+1)$)
\begin{equation*}
\begin{split}
\frac{\w(\B)}{\w(\A)}&=\frac{\prod_i b_i!}{\prod_j a_j!}\frac{|A_k|_{(h(\A))}}{|B_k|_{(h(\B))}}\leq 
d^{h(\B)-h(\A)}\left(\frac{|A_k|}{\frac{n}{2}-d}\right)^{h(\A)}\left(\frac{1}{n/3}\right)^{h(\B)-h(\A)}\\
&\leq n^{-0.49(h(\B)-h(\A))} \left(1+\frac{30k\sqrt{n\log n}}{n/2}\right)^{\sqrt{n}\log^{2/5 }n} 
\\&\leq n^{-0.49(h(\B)-h(\A))} e^{60\log^{0.91}n}\leq n^{-(h(\B)-h(\A))/3},
\end{split}
\end{equation*}
where in the first line we used that $b_i\leq d$ for all $i$ and that $|B_1|\geq \frac{n}{2}-d\geq n/3$, in the second line we used that $|A_k|\leq \frac{n}{2}+d$, (\ref{kupperbound}) and that $d=\lfloor 10k\sqrt{n\log n}\rfloor$, and in the last line we used  (\ref{kupperbound}).
\end{proof}

We further show that if two chains have the same step sizes then their weight decreases with their distance from the middle layer. Given a $k$-chain $\A=\{A_1,\ldots,A_k\}$ with $A_1\subsetneq\ldots\subsetneq A_k$, recall the definition
$$d(\A):=\max\{||A_k|-n/2|, ||A_1|-n/2|\}.$$

\begin{lemma}\label{weightsdecreaselittle}
Given positive integers $a_1,\ldots,a_{k-1}$, let $\A,\B$ be two $k$-chains in $\P(n)$ with step sizes $a_1,\ldots, a_{k-1}$, satisfying $d(\A)>d(\B)$. Then $\w(\A)<\w(\B)$ and in fact $\w(\B)/\w(\A)\geq 1+h(\A)/n$.
\end{lemma}
\begin{proof}
We assume that both chains are downward, the other cases are handled similarly. The weight of a chain $\A$ is given by 
$$\w(\A)=\frac{\prod a_i!}{|A_k|_{(h(\A))}}$$ which, if the $a_i$-s and hence $h(\A)$ are fixed, is a decreasing function of $|A_k|$. The ratio $\w(\B)/\w(\A)$ is bounded below by
$$\frac{\w(\B)}{\w(\A)}\geq \left(\frac{|A_k|}{|B_k|}\right)^{h(\A)}\geq \left(\frac{|B_k|+1}{|B_k|}\right)^{h(\A)}\geq 1+\frac{h(\A)}{n}.$$
\end{proof}

\begin{lemma}\label{weightsdecreaselittle2}Let $a_1,\ldots,a_{k-1}$ and $b_1,\ldots,b_{k-1}$ be positive integers such that $a_i=b_i$ for all but one $i\in [k]$, and if $j$ is the index where the two sequences differ then $b_j=a_j+1$. Suppose that $\A,\B$ are $k$-chains in $\P_{n,d}$ with step sizes $\{a_i\}_{i=1}^{k-1}$ and $\{b_i\}_{i=1}^{k-1}$ respectively such that $d(\A)\leq \lceil (h(\A)+1)/2 \rceil$. Then $\w(\A)\geq \w(\B) n^{1/3}$.
\end{lemma}
\begin{proof}
Without loss of generality we may assume that both chains $\A,\B$ are downward, the proof is similar if one (or both) of them is upward. Note that the condition $d(\A)\leq \lceil (h(\A)+1)/2 \rceil$ means that amongst all chains with step sizes $a_1,\ldots,a_{k-1}$, $\A$ has the minimum distance $d(\A)$. Hence we have $d(\B)\geq d(\A)$ and so $|B_k|\geq |A_k|$. Then by Proposition~\ref{weightformula} we have 
$$\w(\A)=\frac{|A_1|!\prod_i a_i!}{|A_k|!}, \qquad \w(\B)=\frac{|B_1|! b_j \prod_i a_i!}{|B_k|!}.$$
Then, using that $b_j\leq d=\lfloor 10k\sqrt{n\log n}\rfloor$, we have
\begin{equation*}
\begin{split}
\frac{\w(\B)}{\w(\A)}&=\frac{|A_k|_{(h(\A))}b_j}{|B_k|_{(h(\A)+1)}}\leq \frac{b_j}{|B_k|-h(\A)}\left(\frac{|A_k|}{|B_k|}\right)^{h(\A)}\leq \frac{20k\sqrt{n\log n}}{n/4}\cdot 1\leq n^{-1/3}.
\end{split}
\end{equation*}
\end{proof}

Finally we prove in this subsection a weighted supersaturation result for families whose size exceeds $\Sigma(n,k-1)$.
For a family $\F\subseteq \P(n)$  and integers $a_1,\ldots,a_{k-1}\geq 1$, define 
$$\Phi(\F,a_1,\ldots,a_{k-1}):=\{(A_1,\ldots,A_k)\in\F^k: A_1\subsetneq \ldots \subsetneq A_k, \text{ and } |A_{i+1}\setminus A_i|\geq a_i \text{ for all } i\in [k-1]\}.$$
 Now let
$$\W_{a_1,\ldots,a_{k-1}}(\F):=\sum_{(A_1,\ldots,A_k)\in\Phi(\F,a_1,\ldots,a_{k-1})}\w(A_1,\ldots,A_k).$$
Using these definitions we can state the promised supersaturation lemma.
\begin{lemma}\label{supersat}
Let $Q,a_1,\ldots, a_{k-1}$ be positive integers and let $\F\subset \P(n)$ be a family of size $Q$. Then
$$\W_{a_1,\ldots,a_{k-1}}(\F)\geq \W_{a_1,\ldots,a_{k-1}}(\G_Q).$$
\end{lemma}
\begin{proof}
Let $\X_1,\X_2$ be two arbitrary SCDs with chains $A_1,\ldots,A_{N}$ and $B_1,\ldots,B_N$ respectively, where $N=\binom{n}{\lfloor n/2\rfloor}$, and consider the two multisets of integers $\A=\{|A_i\cap \G_Q|:i\in[N]\}$ and $\B=\{|B_i\cap \G_Q|:i\in[N]\}$. Then the two multisets are the same up to permuting their elements.

Let $f(p)$ be the least possible number of $k$-chains with step sizes at least $a_1,\ldots,a_{k-1}$ contained in a chain of length $p$. Then $f(p)$ is exactly equal to the number of $k$-chains with step sizes at least $a_1,\ldots,a_{k-1}$ contained in a \emph{skipless} chain of length $p$. Note that then $f(p+1)-f(p)$ counts the number of $k$-chains with step sizes at least $a_1,\ldots,a_{k-1}$ contained in a skipless chain of length $p+1$ that contain the bottom element of the skipless chain. Hence $f(p+1)-f(p)\geq f(p)-f(p-1)$ for all $p$ and thus $f(p)$ is a convex function of $p$.  Hence every SCD contains at least as many $k$-chains with step sizes at least $a_1,\ldots,a_{k-1}$ from $\F$ as it does from $\G_Q$ where the intersection sizes with the chains are distributed as evenly as possible.

Take a random SCD $\X$ and count the number of $k$-chains with step sizes at least $a_1,\ldots,a_{k-1}$ in $\F$ that are contained in $\X$, call this number $x(\F,\X)$ and similarly define $x(\G_Q,\X)$. Then by the above argument we had for every $\X$ that $x(\F,\X)\geq x(\G_Q,\X)$. Every $k$-chain is contained in $\X$ with probability equal to its weight. Taking expectations we have
$$\W_{a_1,\ldots,a_{k-1}}(\F)={\rm I\kern-.3em E}(x(\F,\X))\geq {\rm I\kern-.3em E}(x(\G_Q,\X))=\W_{a_1,\ldots,a_{k-1}}(\G_Q).$$
\end{proof}

To conclude this subsection we briefly indicate how Lemma~\ref{supersat} implies for example a special case of Theorem~\ref{kleitman}, stating that if a family $\F$ has $\binom{n}{\lfloor n/2\rfloor}+x\leq \Sigma(n,2)$ elements then it contains at least $x\lfloor 1+ n/2\rfloor$ comparable pairs. Indeed if we set $Q=\binom{n}{\lfloor n/2 \rfloor}+x$ then Lemma~\ref{supersat} states that $\W_1(\F)\geq x$. But every comparable pair except for the pair $\{\emptyset,[n]\}$ has weight at most $\lfloor 1+ n/2\rfloor^{-1}$.  Moreover the only comparable pairs of such maximum weight are the ones centered on the two middle layers, hence it is best to take such pairs greedily (i.e.~take those pairs first which have the largest weight). Hence the result follows  if we can show that e.g.~an optimal family cannot contain the empty set.

The above paragraph illustrates some of the main ideas of the proof of the main result. We start with a collection of inequalities given to us by Lemma~\ref{supersat}. We will claim that satisfying these inequalities greedily is the best one can do, assuming the optimal family cannot contain any small sets. The last step is then to show that this is indeed the case, i.e.~if a family contains very small sets then it is bound to contain many more $k$-chains than $\G_M$.

\subsection{Solving Kleitman's conjecture in $\P_{n,d}$}

We are now ready to prove Proposition~\ref{mainmiddle}, in fact we will prove something more. We define the weighted hypergraph $\HH=\HH_{n,d,k}$ to be the $k$-uniform hypergraph on vertex set $V(\HH)=\P_{n,d}$, edges corresponding to $k$-chains in $\P_{n,d}$ and the weight of an edge is given by the weight of the $k$-chain. A function $f:E(\HH)\rightarrow [0,1]$ is called a \emph{measured subhypergraph} of $\HH$ and for an edge $e$ we call $f(e)$ the measure of $e$.

Note that every family $\F$ corresponds to a measured subhypergraph $f_\F$ given by $f(e)=1$ if the $k$-chain $e$ is contained in $\F$, and $f(e)=0$ otherwise. That is, $f_\F$ is the \emph{characteristic function} corresponding to the family $\F$. We say that a measured subhypergraph $f$ is $Q$-\emph{good} if it satisfies the conclusion of Lemma~\ref{supersat}, that is, if for all positive integers $a_1,\ldots,a_{k-1}$ we have
$$\sum_{(A_1,\ldots,A_k)\in\Phi(\P_{n,d},a_1,\ldots,a_{k-1})}\w(A_1,\ldots,A_k)f(A_1,\ldots,A_k)\geq \W_{a_1,\ldots,a_{k-1}}(\G_Q).$$
Note that by Lemma~\ref{supersat} if $\F$ is a family of size at least $M$ in $\P_{n,d}$ then the corresponding characteristic function $f_{\F}$ is $M$-good. The \emph{size} of a measured subhypergraph $f$ is defined as
$$|f|=\sum_{e\in E(\HH)}f(e).$$
Recall the definition of $\Phi^*(\F,\mathbf{a})$ and $<^*_{\F,\mathbf{a}}$ from Section~\ref{setupsection}. For a family $\F\subseteq \P_{n,d}$, a measured subhypergraph $f$ and a vector of positive integers $\mathbf{a}=(a_1,\ldots,a_{k-1})$ denote by $f_{\F,\mathbf{a}}$ the restriction of $f$ to the subhypergraph of $\HH_{n,d,k}$ whose edges are the elements of $\Phi^*(\F,\mathbf{a})$.  We say that $f$ is $(\F,\mathbf{a})$\emph{-compressed} if there is a chain $\A\in\Phi^*(\F,\mathbf{a})$ such that if $\B<^*_{\F,\mathbf{a}}\A$ then $f(\B)=1$ and if $\A<^*_{\F,\mathbf{a}}\B$ then $f(\B)=0$. Similarly define for a family $\F\subseteq \P_{n,d}$, a measured subhypergraph $f$ and vector $\mathbf{a}=(a_1,\ldots,a_{k-1})$ the $(\F,\mathbf{a})$-compression of $f$, which is also a measured subhypergraph, denoted by $c[f,\F,\mathbf{a}]$, as follows. 
\begin{itemize}
\item If $e\notin \Phi^*(\F,\mathbf{a})$ then $c[f,\F,\mathbf{a}](e)=f(e)$. 
\item $c[f,\F,\mathbf{a}]$ is $(\F,\mathbf{a})$-compressed.
\item $|c[f,\F,\mathbf{a}]_{\F,\mathbf{a}}|=|f_{\F,\mathbf{a}}|$.
\end{itemize}
Observe that we always have $|f|=|c[f,\F,\mathbf{a}]|$, i.e.~compression does not change the size of $f$. We say $f$ is \emph{completely compressed} if $f$ is $(\P_{n,d},\mathbf{a})$-compressed for every vector of positive integers $\mathbf{a}=(a_1,\ldots,a_{k-1})$.

\textbf{Example.} Let $n=10$ and $k=2$. Define the families $\F_1=\binom{[n]}{4}\cup \binom{[n]}{6}$, $\F_2=\F_1\cup \binom{[n]}{7}$ and $\F_3=\F_2\cup\binom{[n]}{8}$. Let $\mathbf{a}=(2)$, i.e.~we consider comparable pairs with set difference $2$. Let $f_{\F_1},f_{\F_2},f_{\F_3}$ be the corresponding characteristic functions. Then $f_{\F_1}$ is $(\P(n),\mathbf{a})$-compressed, as the only comparable pairs with set difference $2$ in $\F_1$ are those pairs closest possible to the middle layer, hence of largest weight. Moreover since all comparable pairs in $\F_1$ have set difference $2$, we conclude that $f_{\F_1}$ is completely compressed. Since in $\F_2$ there are no new comparable pairs of set difference exactly $2$, $f_{\F_2}$ is also $(\P(n),\mathbf{a})$-compressed. For $\mathbf{b}=(1)$ however, $f_{\F_2}$ is not $(\P(n),\mathbf{b})$-compressed, as $f_{\F_2}(\{123456,1234567\})=1$ but e.g. $f_{\F_2}(\{1234,12345\})=0$. Similarly $f_{\F_3}$ is not $(\P(n),\mathbf{a})$-compressed as $f_{\F_3}(\{123456,12345678\})=1$ but $f_{\F_3}(12345,1234567)=0$. Note also that for every $Q$ we have that the function $f_{\G_Q}$ corresponding to the centered family $\G_Q$ is completely compressed.

\begin{prop}\label{compressed}
Let $Q>0$ be an integer, $\mathbf{a}=(a_1,\ldots,a_{k-1})$ be a vector of positive integers, $f$ a $Q$-good measured subhypergraph and $\F\subset \P_{n,d}$. Then $c[f,\F,\mathbf{a}]$ is $Q$-good.
\end{prop}
\begin{proof}
We only need to prove that if we denote $\X:=\Phi^*(\F,\mathbf{a})$ then 
$$\sum_{e\in \X}c[f,\F,\mathbf{a}](e)\cdot\w(e)\geq \sum_{e\in \X}f(e)\w(e).$$
This follows from a simple property of the ordering $<^*_{\F,\mathbf{a}}$: note that by the definition of $<^*_{\F,\mathbf{a}}$ and Lemma~\ref{weightsdecreaselittle} we have that if $\A,\B$ are two chains in $\X$ with $d(\A)>d(\B)$ then $\B<^*_{\F,\mathbf{a}}\A$. Hence by Lemma~\ref{weightsdecreaselittle} $c[f,\F,\mathbf{a}]_{\F,\mathbf{a}}$ greedily assigns measure $1$ to the edges in $\X$ of largest weight until it has allocated a total measure equal to $|f_{\F,\mathbf{a}}|$. Since the summation goes over $\X$ both functions in the above inequality can be replaced by their restrictions to $\X$ and then the claim follows from e.g.~the rearrangement inequality\footnote{which states that given numbers $0\leq x_1\leq \ldots \leq x_m$ and $0\leq y_1\leq \ldots\leq y_m$ and a permutation $\pi\in S_m$ we have that $\sum_i x_i y_i \geq \sum_i x_i y_{\pi (i)}$} (see e.g.~\cite{inequalitieshardy},  Section~$10.2$, Theorem~$368$).
\end{proof}

Instead of proving Proposition~\ref{mainmiddle} directly we will show the following stronger statement. As is often the case, the stronger statement will be easier and more natural to prove.
\begin{prop}\label{mainmiddlestrong}
Amongst all $M$-good measured subhypergraphs, $f_{\G_M}$ has the smallest size.
\end{prop}
\begin{proof}
The collection of $M$-good measured subhypergraphs forms a closed subset of the compact set $[0,1]^{E(\HH)}$, so the restriction of $|\cdot|$ to this subset attains its minimum. Hence it suffices to show that for any $M$-good $f$ we have either $|f|=|f_{\G_M}|$ or we can find an $M$-good $f'$ with $|f'|<|f|$.  Recall that $\Phi^*(\P_{n,d},a_1,\ldots,a_{k-1})$ is defined to be the collection of all $k$-chains with step sizes precisely $a_1,\ldots,a_{k-1}$ contained in $\P_{n,d}$. By Proposition~\ref{compressed} it suffices to consider those measured subhypergraphs which are completely compressed.

Let $g$ be an $M$-good measured subhypergraph. For a list of positive integers $\mathbf{a}=(a_1,\ldots,a_{k-1})$ write $g_{\mathbf{a}}$ for the restriction of $g$ to the set $\Phi^*(\P_{n,d},\mathbf{a})$, and similarly let $f_{\mathbf{a}}$ be the restriction of $f_{\G_M}$ to the same set. Let $p$ be the smallest positive integer for which there exist positive integers $a_1,\ldots, a_{k-1}$ with $\sum_i a_i=p$ such that $|g_{\mathbf{a}}|>|f_{\mathbf{a}}|$. We split into two cases according to whether such a $p$ exists or not. 

{\bf Case 1:} If such a $p$ exists then pick $a_1,\ldots,a_{k-1}$ with $\sum_i a_i=p$ and $|g_{\mathbf{a}}|>|f_{\mathbf{a}}|$.  Note that both $g_{\mathbf{a}}$ and $f_{\mathbf{a}}$ are $(\P_{n,d},\mathbf{a})$-compressed: $g_{\mathbf{a}}$ is because as said before, by Proposition~\ref{compressed} it suffices to consider completely compressed measured subhypergraphs, and $f_{\mathbf{a}}$ is because of how we defined $<^*_{\P_{n,d},\mathbf{a}}$. Note that this implies that $g_\mathbf{a}(e)\geq f_{\mathbf{a}}(e)$ for all $e\in E(\HH)$, and there exists at least one $e^*\in E(\HH)$ such that $g_{\mathbf{a}}(e^*)>f_{\mathbf{a}}(e^*)$.
Let $\eps':=g(e^*)-f(e^*).$

Define the following collection of $(k-1)$-sequences obtained from $a_1,\ldots,a_{k-1}$ by decreasing one of the $a_i$'s by one, assuming $a_i\neq 1$:  $$\A_{\mathbf{a}}:=\{(a_1,a_2,\ldots,a_{i-1},a_i-1,a_{i+1},\ldots,a_{k-1}): i\in [k-1], a_i\geq 2\}.$$
Observe that by the choice of $p$, for every $\mathbf{b}=(b_1,\ldots,b_{k-1})\in\A_{\mathbf{a}}$ and for every $e\in E(\HH)$ that corresponds to a $k$-chain with step sizes exactly $b_1,\ldots,b_{k-1}$ we have $g(e)\leq f_{\G_M}(e)$. Now for every $\mathbf{b}\in\A_{\mathbf{a}}$ pick an $e_{\mathbf{b}}\in E(\HH)$ of largest possible weight that corresponds to a $k$-chain with step sizes exactly $b_1,\ldots,b_{k-1}$ and $g(e_{\mathbf{b}})=0$ and denote the collection of these at most $k-1$ edges by $F$. Choosing such edges is possible since $\G_M$ is contained in $\P_{n,r}$ and $r\ll d$.

Define a measured subhypergraph $g'$ as follows.
\[
g'(e)=
\left\{
\begin{array}
{ll}
\frac{\eps'}{2k}  & : e\in F, \\
g(e)-\eps'  & : e=e^*, \\
g(e) & \text{otherwise.}
\end{array}
\right.
\]
Observe that $$|g'|=|g|-\eps' + |F|\frac{\eps'}{2k}\leq |g|-\eps'/2<|g|,$$
hence (recalling the first paragraph of this proof) it suffices to show that $g'$ is $M$-good. Pick any positive integers $b_1,\ldots,b_{k-1}$, and we will show that
\begin{equation}\label{supersatweneed}\sum_{(A_1,\ldots,A_k)\in\Phi(\P_{n,d},b_1,\ldots,b_{k-1})}\w(A_1,\ldots,A_k)g'(A_1,\ldots,A_k)\geq \W_{b_1,\ldots,b_{k-1}}(\G_M).
\end{equation}
 If for some $i$ we have $b_i>a_i$ then the changes we have made to $g$ did not affect this inequality, and
since $g$ was $M$-good, (\ref{supersatweneed}) still holds for $g'$. If $b_i=a_i$ for all $i\in [k-1]$ then (\ref{supersatweneed}) holds by definition of $\eps'$. Now suppose that there exists some $j\in [k-1]$ such that $b_j\leq a_j-1$. Let $e_j\in F$ be the edge defined for the sequence $(a_1,\ldots,a_{j-1},a_j-1,a_{j+1},\ldots,a_{k-1})$  above. If $h(e^*)\leq \sqrt{n}\log^{1/5}n$ then by Proposition~\ref{weightsdecreasealot} we have $\w(e^*)n^{1/3}\leq \w(e_j)$. If $h(e^*)\geq \sqrt{n}\log^{1/5}n$ then by (\ref{rupperbound}) we may take the $e_j$-s to have as small $d(e_j)$ as possible (and hence maximising their weight by Lemma~\ref{weightsdecreaselittle}) since none of the chains of height at least $\sqrt{n}\log^{1/5}n$ are present in $\G_M$, and by Lemma~\ref{weightsdecreaselittle2} we also have $\w(e^*)n^{1/3}\leq \w(e_j)$. So
$$\sum_{(A_1,\ldots,A_k)\in\Phi(\P_{n,d},b_1,\ldots,b_{k-1})}\w(A_1,\ldots,A_k)\left(g'(A_1,\ldots,A_k) - g(A_1,\ldots,A_k)\right)\geq  \frac{\eps'}{2k}\w(e_j) - \eps' \w(e^*)>0.$$
Since $g$ was $M$-good we conclude that $g'$ also satisfies (\ref{supersatweneed}) and so is $M$-good. This completes the proof of the first case.

{\bf Case 2:} For the second case we suppose such a $p$ does not exist, i.e.~for every list of positive integers $\mathbf{a}=(a_1,\ldots,a_{k-1})$ we have $|g_{\mathbf{a}}|\leq |f_{\mathbf{a}}|$. We claim that then $|g_{\mathbf{a}}|= |f_{\mathbf{a}}|$ for all sequences $a_1,\ldots,a_{k-1}$ and this will finish the proof as then $|g|=|f_{\G_M}|$.   
Suppose this is not true and let $q$ be the largest positive integer such that there exists a list of integers $\mathbf{a}=(a_1,\ldots,a_{k-1})$ with $\sum_i a_i=q$ and $|g_{\mathbf{a}}|< |f_{\mathbf{a}}|$. Pick such an $\mathbf{a}$. Note that by the choice of $q$ and since $g$ is completely compressed we have that if $\mathbf{b}=(b_1,\ldots,b_{k-1})$ is a list such that $a_i\leq b_i$ for all $i\in [k-1]$ and $e$ is any edge then $g_{\mathbf{b}}(e)= f_{\mathbf{b}}(e)$. Moreover since $|g_{\mathbf{a}}|< |f_{\mathbf{a}}|$ there exists an edge $e^*\in \Phi^*(\P_{n,d},\mathbf{a})$ such that $g_{\mathbf{a}}(e^*)<f_{\mathbf{a}}(e^*)$.  We have
\begin{equation*}
\begin{split}
\sum_{\A\in\Phi(\P_{n,d},\mathbf{a})}\w(\A)g(\A)&=
\sum_{\A\in\Phi^*(\P_{n,d},\mathbf{a})}\w(\A)g(\A) + \sum_{\mathbf{b}:\mathbf{b}>\mathbf{a}} \sum_{\A\in\Phi^*(\P_{n,d},\mathbf{b})}\w(\A)g(\A)
\\&=\sum_{\A\in\Phi^*(\P_{n,d},\mathbf{a})}\w(\A)g(\A) + \sum_{\mathbf{b}:\mathbf{b}>\mathbf{a}} \sum_{\A\in\Phi^*(\P_{n,d},\mathbf{b})}\w(\A)f_{\G_M}(\A)
\\&<\sum_{\A\in\Phi^*(\P_{n,d},\mathbf{a})}\w(\A)f_{\G_M}(\A) + \sum_{\mathbf{b}:\mathbf{b}>\mathbf{a}} \sum_{\A\in\Phi^*(\P_{n,d},\mathbf{b})}\w(\A)f_{\G_M}(\A)
\\&=\sum_{\A\in\Phi(\P_{n,d},\mathbf{a})}\w(\A)f_{\G_M}(\A)
=\sum_{\A\in\Phi(\G_M,\mathbf{a})}\w(\A)=\W_{\mathbf{a}}(\G_M).
\end{split}
\end{equation*}
Hence by Lemma~\ref{supersat} the measured subhypergraph $g$ is not $M$-good, contradicting our assumptions. This completes the proof of Proposition~\ref{mainmiddlestrong}.
\end{proof}

\begin{proof}[Proof of Proposition~\ref{mainmiddle}]
Let $\F\subseteq \P_{n,d}$ be a family of size $M$. Then $f_\F$ is $M$-good by Lemma~\ref{supersat}, hence by Proposition~\ref{mainmiddlestrong} we have $|f_{\G_M}|\leq|f_\F|$, implying by definition that $\F$  contains at least as many $k$-chains as $\G_M$.
\end{proof}

\subsection{Non-centered families in $\P_{n,d}$}
In the previous subsections we have shown that amongst families contained in $\P_{n,d}$, centered families are the best (i.e.~given the size they minimize the number of $k$-chains). In the next section our goal will be to show that an optimal family cannot contain sets from outside of $\P_{n,d}$. For that we will make use of a lemma stating that if a family of size $M$ is contained in $\P_{n,d}$, but misses some number of elements from the middle layers (and hence it is not centered) then this family contains significantly more $k$-chains than $\G_M$. This technique was used by Das--Gan--Sudakov~\cite{dasgansuda} to prove Theorem~\ref{dasgansudakovthm}.

Let $\C\subset \P_{n,r-1}$ be a family of size at most $\binom{n}{\lfloor(n+r)/2\rfloor}$. Write $\P':=\P_{n,d}\setminus \C$ and say that a measured subhypergraph $f$ is contained in $\P'$ if it assigns zero to every $k$-chain that intersects $\C$.  Define the measured hypergraph $\hat{f}_{\C,M}$, contained in $\P'$, as follows. 
\begin{itemize}
\item $\sum_{e\in \Phi^*(\P',\mathbf{a})}\hat{f}_{\C,M}(e)\w(e)=\sum_{e\in \Phi^*(\P_{n,d},\mathbf{a})}f_{\G_M}(e)\w(e)$ for all $\mathbf{a}$, and
\item $\hat{f}_{\C,M}$ is $(\P',\mathbf{a})$-compressed for all $\mathbf{a}$. 
\end{itemize}
That is, $\hat{f}_{\C,M}$ is obtained by greedily taking edges of largest possible weights, avoiding $\C$, to satisfy the definition of being $M$-good. Note that the first equality in the above definition of $\hat{f}_{\C,M}$ can be satisfied because $r\ll d$, and that $\hat{f}_{\C,M}$ is $M$-good by definition.

\begin{prop}\label{mainmiddlestrong2}
Let $0\leq t\leq \binom{n}{\lfloor(n+r)/2\rfloor}$ and let $\C$ be a family of size $t$ contained in $\P_{n,d}$. If $g$ is an $M$-good measured subhypergraph contained in $\P'=\P_{n,d}\setminus \C$ then $|g|\geq |\hat{f}_{\C,M}|$.
\end{prop}
\begin{proof}
The proof of this proposition will be essentially the same as the proof of Proposition~\ref{mainmiddlestrong}, therefore we only give a sketch.  By Proposition~\ref{compressed} we may assume that $g$ is $(\P',\mathbf{a})$-compressed for every list $\mathbf{a}=(a_1,\ldots,a_{k-1})$. For ease of notation, write $f:=\hat{f}_{\C,M}$ and as before,  for a list of positive integers $\mathbf{a}=(a_1,\ldots,a_{k-1})$ write $g_{\mathbf{a}}$ for the restriction of $g$ to the set $\Phi^*(\P',\mathbf{a})$, and similarly let $f_{\mathbf{a}}$ be the restriction of $f$ to the same set. Let $p$ be the smallest positive integer for which there exist positive integers $a_1,\ldots, a_{k-1}$ with $\sum_i a_i=p$ such that $|g_{\mathbf{a}}|>|f_{\mathbf{a}}|$. We split into two cases according to whether such a $p$ exists or not. If such a $p$ exists then we can find an $M$-good measured subhypergraph $g'$ contained in $\P'$ with $|g'|<|g|$ the same way as we did in the proof of Proposition~\ref{mainmiddlestrong}. If such a $p$ does not exist then we may choose the largest positive integer $q$ such that there exists a list of integers $\mathbf{a}=(a_1,\ldots,a_{k-1})$ with $\sum_i a_i=q$ and $|g_{\mathbf{a}}|< |f_{\mathbf{a}}|$. The existence of such $q$ would show that $g$ is not $M$-good and also result in a contradiction in the same fashion as in Proposition~\ref{mainmiddlestrong}, hence we conclude that $|g_{\mathbf{a}}|=|f_{\mathbf{a}}|$ for all $\mathbf{a}$ and hence $|g|=|f|$.
\end{proof}

\section{Excluding very small and very large sets}\label{smalllargesetssection}

In this section we show that an optimal family cannot contain sets from $\P(n)\setminus \P_{n,d}$. The main ideas in this section are similar to ideas in the work of Das--Gan--Sudakov~\cite{dasgansuda}. For any $j$ let $\HH_{j,\ell}$ be the $\ell$-uniform hypergraph with vertex set $V(\HH_{j,\ell})=\P_{n,j}$, and edges corresponding to $\ell$-chains. Denote $\Delta_{j,\ell}$ the maximum degree of $\HH_{j,\ell}$.

We continue our train of thought from the previous section with the following proposition:
\begin{prop}\label{tmissing}
Let $0\leq t\leq \binom{n}{\lceil \frac{n+d-1}{2}\rceil}$ and let $\C$ be a family of $t$ elements contained in $\P_{n,r-2}$. Let $s=\sum_{v\in \C}d(v,\HH_{r-2,k})$ be the sum of the degrees of vertices in $\C$ in $\HH_{r-2,k}$. If $\F\subset \P_{n,d}\setminus \C$ is a family of size $M$ then $c_k(\F)\geq c_k(\G_M)+\xyz$.
\end{prop}
\begin{proof}
By Proposition~\ref{mainmiddlestrong2} we have that $c_k(\F)\geq |\hat{f}_{\C,M}|$. Since $c_k(\G_M)=|f_{\G_M}|$  it suffices to show that $|\hat{f}_{\C,M}|-|f_{\G_M}|\geq \xyz$. Let $E$ be the collection of $k$-chains contained in $\P_{n,r-2}$ that intersect $\C$. Note that every element $e\in E$ is present in $\G_M$ but missing from $\F$, and in fact we have $\hat{f}_{\C,M}(e)=0$ and $f_{\G_M}(e)=1$. Moreover since  $\P_{n,r-1}\subseteq \G_M$, every $e\in E$ had to be replaced by edges of strictly smaller weight in $\hat{f}_{\C,M}$. By Lemma~\ref{weightsdecreaselittle} we have that $|\hat{f}_{\C,M}|\geq |f_{\G_M}|+ \frac{1}{n}\cdot |E|$. Since $|E|\geq s/k$ we get the required result.
\end{proof}

Let $A$ be a set in $\P_{n,j}$ for some $j\geq k$, and let $v$ be the vertex corresponding to $A$ in $\HH_{j,k}$. We wish to estimate the degree $d(v,\HH_{j,k})$ of $v$ in $\HH_{j,k}$. Denote the smallest and largest elements' sizes of $\P_{n,j}$ by $p_-$ and $p_+$, thats is,  $p_-=\lceil \frac{n-j+1}{2}\rceil$ and $p_+= \lceil\frac{n+j-1}{2}\rceil$. For $q\in [k]$ let $$S_q=\{\mathbf{a}=(a_1,\ldots,a_{k-1}):~ a_1+\ldots + a_{q-1}\leq |A|-p_-, ~ a_{q}+\ldots + a_{k-1}\leq p_+-|A|\}. $$ Then
$$d(v,\HH_{j,k})=\sum_{q=1}^k \sum_{\mathbf{a}\in S_q}\frac{|A|_{(a_1+\ldots +a_{q-1})} (n-|A|)_{(a_q+\ldots + a_{k-1})}} {\prod a_i!}.$$
The largest term in the second sum occurs when the enumerator has $j$ terms and the denominator is as small as possible, i.e.~when $\mathbf{a}$ is such that all $a_i\in\{\lfloor (j-1)/(k-1)\rfloor, \lceil (j-1)/(k-1) \rceil\}$ and $\sum a_i=j-1$. Let $\mathbf{a^*}=(a^*_1,\ldots,a^*_{k-1})$ be such an $\mathbf{a}$. Since $|S_q|\leq n^{k-1}$ we get
$$d(v,\HH_{j,k})\leq kn^{k-1} \frac{|A|_{(|A|-p_-)} (n-|A|)_{(p_+-|A|)}}{\prod a^*_i!}.$$
This implies that 
$$\frac{\lceil\frac{n+j-1}{2}\rceil_{(j-1)}}{\prod a^*_i!}\leq\Delta_{j,k}\leq n^k \frac{\lceil\frac{n+j-1}{2}\rceil_{(j-1)}}{\prod a^*_i!},$$
where the lower bound comes from simply counting the number of chains with step sizes precisely $\mathbf{a}^*$ containing a fixed set of size $p_+$.
Suppose $A$ is such that there exists an $\mathbf{a}$ and a $q\in [k-1]$ such that $a_1+\ldots +a_{q-1}=|A|-p_-$ and $\sum a_i = j-1$ and all $a_i\in \{\lfloor (j-1)/(k-1)\rfloor, \lceil (j-1)/(k-1) \rceil\}$. Then for $j\leq r$ we get for the corresponding $v$ that
\begin{equation}\label{degreesarebehaving}
d(v,\HH_{j,k})\geq \frac{(n/2)_{(n/2-p_-)}(n/2)_{(p_+-n/2)}}{\prod a^*_i!}\geq \Delta_{j,k}n^{-k}\left(\frac{p_-}{p_+}\right)^{r/2}\geq \Delta_{j,k}n^{-k}\left(1-\frac{r}{n/3}\right)^r\geq \Delta_{j,k}n^{-k-1}.
\end{equation}
We now show that a small change in $j$ does not change the degrees by much. Let $\mathbf{a}^{**}$ be such that all $a^{**}_i\in\{\lfloor (j)/(k-1)\rfloor, \lceil (j)/(k-1) \rceil\}$ and $\sum a^{**}_i=j$. Then
\begin{equation}\label{degreesarebehaving2}
\Delta_{j,k}\geq \frac{\lceil\frac{n+j-1}{2}\rceil_{(j-1)}}{\prod a^*_i!}\geq \frac{\lceil\frac{n+j+1}{2}\rceil_{(j)}}{\prod a^{**}_i!} n^{-1}\geq \Delta_{j+1,k}n^{-k-1}.
\end{equation}
Equipped with these bounds we are now ready to tackle the main result of this section.

\begin{prop}\label{doesntcontainsmalllargesets}
If $\F\subset \P(n)$ is a family of size $M$ with $\F\setminus \P_{n,d}\neq \emptyset$ then $c_k(\F)> c_k(\G_M)$.
\end{prop}
\begin{proof}
 Let $M':=|\F\cap \P_{n,d}|$ and define $r'$ such that $\Sigma(n,r'-1)<M'\leq \Sigma(n,r')$. Set $$b_+=\left\lceil\frac{n+r'-1}{2}\right\rceil, \qquad b_-=\left\lceil\frac{n-r'+1}{2}\right\rceil, \qquad  c_+=\left\lceil\frac{n+d+1}{2}\right\rceil,  \qquad c_-=\left\lceil\frac{n-d-1}{2}\right\rceil.$$ Note that $\P_{n,r'}=\{A\in\P(n):b_-\leq|A|\leq b_+\}$ and $\P(n)\setminus \P(n,d)=\{A\in\P(n):|A|\leq c_- \text{ or } |A|\geq c_+\}$. As 
$$\binom{n}{\leq c_-}+\binom{n}{\geq c_+} \ll \binom{n}{b_+},$$ we have $r'\in\{r-1,r\}$. Recall that by (\ref{rkplus1}) we have $r\geq k+2$ and so $r'\geq k+1$. We will assume throughout the proof that $r'\geq k+2$. The proof for the case $r'=k+1$ is very similar (in fact easier), but needs to be handled separately - we will do so later.

Let $\S$ be the family of those sets $A\in \P_{n,r'-2}$ for which there exists an $\mathbf{a}=(a_1,\ldots, a_{k-1})$ satisfying $\sum a_i = r'-3$ with $a_i\geq 1$ for all $i$, and there exists a $q\in [k]$ with $b_- + a_1 + \ldots + a_{q-1} = |A|$ and moreover $a_i\in\{\lfloor (r'-3)/(k-1)\rfloor,\lceil (r'-3)/(k-1)\rceil\}$. Note that $\S$ consists of at least $k$ complete layers in $\P_{n,r'-2}$ (corresponding to splitting up the distance between $b_-$ and $b_+$ into $k-1$ roughly equal pieces). Observe that we used the fact that $r'\geq k+2$ here.

 Let $\A:=\P_{n,r'-2}\setminus \F$. For $j\in I=\{0,\ldots,c_-\}\cup \{c_+\ldots,n\}$ let $R_j=\F\cap\binom{[n]}{j}$ and let $h_j$ denote the number of $k$-chains in $\F$ which contain an element of $R_j$ and $k-1$ elements from $\P_{n,r'-2}$.  Hence we have by Proposition~\ref{tmissing} that
$$c_k(\F)\geq c_k(\G_{M'})+\frac{\sum_{v\in \A}d(v,\HH_{r'-2,k})}{kn}+\sum_{j\in I} h_j.$$ Note that $c_k(\G_{M})-c_k(\G_{M'})\leq \left(M-M'\right)\Delta_{r,k}$, so it suffices to show that
\begin{equation}\label{sudastyleineq}\frac{\sum_{v\in \A}d(v,\HH_{r'-2,k})}{kn}+\sum_{j\in I}h_j>(M-M')\Delta_{r,k}=\sum_{j\in I}|R_j|\Delta_{r,k}.
\end{equation}
 W.l.o.g.~we assume that $\sum_{j\leq c_-}|R_j|\geq \sum_{j\geq c_+}|R_j|$, the proof otherwise is identical. From now on we always assume $j\in [c_-]$, the extra factor of $2$ will be dominated by larger terms in our inequalities. Define $\beta$ by
$$\beta \binom{n}{\leq c_-}=\sum_{j\leq c_-}|R_j|.$$
Now we split into two cases. For the first case assume that $|\S\setminus \F|\geq \beta\binom{n}{b_-+1}/n^5$. Then by~(\ref{degreesarebehaving}) and (\ref{degreesarebehaving2}) we get
$$\frac{\sum_{v\in \A}d(v,\HH_{r'-2,k})}{kn}\geq \beta\binom{n}{b_-+1}    \Delta_{r'-2,k}  n^{-k-10}\geq \beta\binom{n}{b_-+1}\Delta_{r,k}n^{-10k}.$$
Now note that
$$\binom{n}{\leq c_-}\leq n^{-10k^2}\binom{n}{b_-+1},$$ 
and hence (\ref{sudastyleineq}) holds in this case. 

Henceforth we assume $|\S\setminus \F|\leq \beta\binom{n}{b_-+1}/n^5$.  Let $\T$ be the family of those sets in $\binom{[n]}{b_-+1}$ which are not contained in any $(k-1)$-chains in $\F\cap \P_{n,r'-2}$. In other words, if $A\in\T$ then every $(k-1)$-chain in $\S$ containing $A$ intersects $\S\setminus \F$. Recall that $\S$ contains at least $k$ complete layers and let $\S'$ denote the bottom $k-1$ layers from $\S$, so that $\S'$ contains all sets of sizes $b_-+1=s_1 < s_2 <\ldots < s_{k-1} \leq b_+-1$. For all $i\in [k-1]$, write $\Q_i := (\S'\setminus \F)\cap \binom{[n]}{s_i}$. Let $\T_1:=\T\setminus \Q_1$ and for $i\in [k-1]\setminus \{1\}$ define $\T_i :=\partial (\T_{i-1},s_i)\setminus \Q_i$, where $\partial(\T_{i-1},s_i)$ denotes the family of sets $A\in\binom{[n]}{s_i}$ for which there exists a set $B\in \T_{i-1}$ such that $B\subset A$ (i.e.~the upper shadow of $\T_{i-1}$ on level $s_i$). Since every $(k-1)$-chain in $\S'$ that intersects $\T$ has to intersect $\S'\setminus \F$, we conclude that $\T_{k-1}=\emptyset$. For all $i\in [k-1]$ define $q_i := |\Q_i|\binom{n}{s_i}^{-1}$ and similarly $t_i:=|\T_i|\binom{n}{s_i}^{-1}$. By the normalized matching property\footnote{In our context this means that (for all $p$), whenever $\C\subseteq\binom{[n]}{p}$ and $\C^*$ is the subset of $\binom{[n]}{p+1}$ consisting of the elements of $\binom{[n]}{p+1}$ covering elements of $\C$, it holds that $|\C|\binom{n}{p}^{-1}\leq |\C^*|\binom{n}{p+1}^{-1}$.} of the Boolean lattice we have the following inequalities:
\begin{itemize}
\item $t_i\leq q_{i+1} + t_{i+1}$ for all $i\in [k-3]$, and
\item $t_{k-2}\leq q_{k-1}$.
\end{itemize}
By summing up all these inequalities we conclude that $t_1 + q_1\leq q_1 + q_2 + \ldots + q_{k-1}$, which since $s_{k-1}\leq b_+-1$ implies that $|\T|\leq 3\left(|\Q_1|+|\Q_2|+\ldots + |\Q_{k-1}|\right)=3|\S'\setminus \F|\leq 3|\S\setminus \F|\leq \beta\binom{n}{b_-+1}/n^4$.

Now we have the bound
\begin{equation}\label{referringforjozsi}
h_j \geq |R_j|\binom{n-j}{b_-+1-j}-|\T|\binom{b_-+1}{j}.
\end{equation}
For $j\in [c_-]$ define $\beta_j$ by $|R_j|=\beta_j\binom{n}{j}$. Using $\binom{n}{b_-+1}\binom{b_-+1}{j}=\binom{n}{j}\binom{n-j}{b_-+1-j}$ and that $h_j\geq 0$ we get
\begin{equation*}
\begin{split}
h_j &\geq \max\left\{0,\beta_j\binom{n}{j}\binom{n-j}{b_-+1-j}-\frac{\beta}{n^4}\binom{n}{j}\binom{n-j}{b_-+1-j}\right\}\\ &\geq \max\left\{0,\beta_j\binom{n}{j}\binom{n-c_-}{b_-+1-c_-}-\frac{\beta}{n^4}\binom{n}{j}\binom{n-c_-}{b_-+1-c_-}\right\}.
\end{split}
\end{equation*}
Since $\sum_{j\leq c_-}\beta_j\binom{n}{j}=\beta\binom{n}{\leq c_-}$ we have
$$\sum_{j\leq c_-}h_j\geq \binom{n-c_-}{b_-+1-c_-}\left( \sum_{j\leq c_-} \binom{n}{j}\beta_j - \sum_{j\leq c_-}\binom{n}{j}\frac{\beta}{n^4}\right)\geq \frac{1}{4}\binom{n-c_-}{b_-+1-c_-}\beta\binom{n}{\leq c_-}.$$
To complete the proof it only remains to show that $\binom{n-c_-}{b_-+1-c_-}\gg \Delta_{r,k}$, as then (\ref{sudastyleineq}) holds. Note that $\Delta_{r,k}\leq n^{k+r}$ - indeed, there are at most $n^k$ ways to choose the sizes of the $k$ sets in a $k$-chain, and there are at most $n^r$ distinct $r$-chains through a fixed set in $\P_{n,r}$. Moreover we have
$$\binom{n-c_-}{b_-+1-c_-}\geq \binom{n/2}{4k\sqrt{n\log n}}\geq n^{k\sqrt{n\log n}}\geq n^{2(k+r)},$$
and the proof is complete.
\end{proof}

All that is missing now is the case $r'=k+1$ of Proposition~\ref{doesntcontainsmalllargesets}.
Fortunately when $r'=k+1$ we can directly apply the results of Das--Gan--Sudakov~\cite{dasgansuda}. Recall that $M'=|\F'|$.
\begin{thm}[\label{dasgansudaweak}Corollary of Theorem $4.2$ of \cite{dasgansuda}]
Let $\F'\subset \P_{n,d}$ be a family of size $\Sigma(n,k)\leq |\F'|\leq \Sigma(n,k+1)$ with at least $t$ sets missing from the middle $k-1$ levels. Then
$$c_k(\F')\geq c_k(\G_{M'})+\frac{t}{n}\Delta_{k+1,k}.$$
\end{thm}
\begin{proof}[Proof of Proposition~\ref{doesntcontainsmalllargesets} in the case $r'=k+1$]
We follow the notation of the proof of Proposition~\ref{doesntcontainsmalllargesets}. W.l.o.g.~assume that $\sum_{j\leq c_-}|R_j|\leq \sum_{j\geq c_+}|R_j|$, the proof otherwise is identical. Similarly to the proof of Proposition~\ref{doesntcontainsmalllargesets} define $h_j$ to be the number of $k$-chains in $\F$ which contain an element of $R_j$ and $k-1$ elements from $\P_{n,k-1}$. Setting $\F':=\F\cap\P_{n,d}$ and $t:=|\P_{n,k-1}\setminus \F'|$ and applying Theorem~\ref{dasgansudaweak} we get that
$$c_k(\F)\geq c_k(\G_{M'})+\frac{t}{n}\Delta_{k+1,k}+\sum_{j\geq c_+}h_j,$$
and hence as before it suffices to show that
$$\frac{t}{n}\Delta_{k+1,k}+\sum_{j\geq c_+}h_j\geq \left(M-M'\right)\Delta_{k+1,k}.$$
If $t\geq n\left(M-M'\right)$ then this inequality holds as each $h_j$ is non-negative, hence we may assume $t\leq n\left(M-M'\right)\leq 2n\sum_{j\geq c_+}|R_j|$. In this case we will in fact show that
$$\sum_{j\geq c_+}h_j\geq \left(M-M'\right)\Delta_{k+1,k}.$$ 
Following the notation of~\cite{dasgansuda}, let $a=\lceil\frac{n+k}{2}\rceil$ so that the $k-1$ middle levels are those sets of sizes between $a-k+1$ and $a-1$. As before in (\ref{referringforjozsi}), for $j\geq c_+$ we have the lower bound
$$h_j\geq \max\bigg\{|R_j|\binom{j}{a-1}-t\binom{n-a+1}{j-a+1},0\bigg\}\binom{a-1}{k-2}(k-2)!.$$
Now observe that for $j\geq c_+$ we have
$$\binom{j}{a-1}\binom{n-a+1}{j-a+1}^{-1}=\frac{j!(n-j)!}{(a-1)!(n-(a-1))!}\geq \left(1+\frac{10k\sqrt{n\log n}}{n}\right)^{4k\sqrt{n\log n}}\geq n^{20k^2}.$$
Hence it suffices to show
$$\sum_{j\geq c_+}\max\bigg\{|R_j|-\frac{\sum_{i\geq c_+}|R_i|}{n^{19k^2}},0\bigg\}\binom{j}{a-1}\binom{a-1}{k-2}(k-2)!\geq 2\Delta_{k+1,k}\sum_{j\geq c_+}|R_j|.$$
Now since
$$\Delta_{k+1,k}=\left(\binom{a}{k-1}+\binom{a}{k}\binom{k}{2}\right)(k-1)!\leq n^5\binom{a-1}{k-2}(k-2)!,$$
and since for every $j\geq c_+$ we have $\binom{j}{a-1}\geq n^{10}$, it is enough to show
$$n^4\sum_{j\geq c_+}\left(|R_j|-\frac{\sum_{i\geq c_+}|R_i|}{n^{19k^2}}\right)\geq \sum_{j\geq c_+}|R_j|.$$
However the left hand side is at least
$$n^4\sum_{j\geq c_+}\left(|R_j|-\frac{\sum_{i\geq c_+}|R_i|}{n^{19k^2}}\right)\geq n^4\sum_{j\geq c_+}|R_j| - \frac{n^5}{n^{19k^2}}\sum_{j\geq c_+}|R_j|\gg \sum_{j\geq c_+}|R_j|,$$
and the proof is complete.
\end{proof}

\section{Proof of Theorem~\ref{mainthm}}
Let $\F$ be a family of size $M$. If $\F\not\subset\P_{n,d}$ then by Proposition~\ref{doesntcontainsmalllargesets} we have $c_k(\F) > c_k(\G_M)$. If on the other hand $\F\subseteq\P_{n,d}$ then by Proposition~\ref{mainmiddle} we have $c_k(\F)\geq c_k(\G_M)$. Hence $\G_M$ minimizes the number of contained $k$-chains amongst families of size $M$ in $\P(n)$, and the proof is complete.

\section{Open problems}

The main open problem that remains to be solved is of course Kleitman's conjecture, Conjecture~\ref{kleitmanconjecture}. Observe that throughout this paper we heavily relied on the fact that $n\gg k$, and most of the methods would break down if $k$ was allowed to be comparable to $n$. It seems that new ideas are needed to tackle these cases, and any partial results on this problem are likely to get us closer to solving Kleitman's conjecture in its full generality.
It would be interesting to have a proof of Kleitman's conjecture for large $n$, under the assumption  $k\geq n/10$ (or even $k \geq n-100$).

Another direction that might be of interest is to extend the question of minimizing the number of $k$-chains to other posets, hence generalizing Kleitman's question. Instead of considering families in $\P(n)=\{0,1\}^n$ one could ask the same questions for $[m]^2$ or even $[m]^d$. A $k$-chain in $[m]^d$ is a set of $k$ distinct points satisfying $\mathbf{a}_1\leq \ldots \leq \mathbf{a}_k$ (where $\mathbf{b}\leq \mathbf{c}$ means $b_i\leq c_i$ for all $i\in [d]$). Solving the following problem in full generality seems hopeless (in particular it contains Kleitman's question as a special case where $m=2$), but partial results for larger $m$ would be of much interest. Is a similar phenomenon as in Kleitman's conjecture likely to hold for these posets as well?

\begin{proble}\label{generalquestion}
Given $d,m,M,k$, which sets $\F\subseteq [m]^d$ of size $|\F|=M$ minimize the number of $k$-chains?
\end{proble}

Consider the following definition of an $m$\emph{-centered} set: a set $\F\subseteq [m]^d$ is $m$-centered if for all $\mathbf{a},\mathbf{b}\in [m]^d$ with $\mathbf{a}\in \F$ and $\mathbf{b}\notin \F$ we have that $$\left\lvert\sum_{i=1}^d a_i - \frac{dm}{2}\right\rvert \leq \left\lvert\sum_{i=1}^d b_i - \frac{dm}{2}\right\rvert,$$ and in case of equality we have $\sum a_i \geq \sum b_i$. Note that taking $m=2$ we get our usual definition of centered families. The following conjecture is not much more than a natural guess, as we have little evidence supporting it. Once again we do not make the (false) claim that $m$-centered sets are the \emph{only} ones minimizing the number of $k$-chains.  

\begin{conj}
Given $m$ there exists a number $d_0(m)$ such that if $d\geq d_0(m)$ then the answer to Problem~\ref{generalquestion} is given by $m$-centered sets.
\end{conj}

Note that if we do not assume $d$ to be large enough then this natural conjecture might fail. One small counterexample is given by the case $m=16, d=2, k=2$ where the family $\F:=\{(a_1,a_2)\in [16]^2 : |a_1+a_2-16|\leq 5\}$ can be improved by letting $\F' := \F \setminus \{(5,6)\} \cup \{(10,0)\}$.

Instead of the poset $\{0,1\}^n$ we can consider the poset $[0,1]^n$. Given a subset $\F\subseteq [0,1]^n$ let $\C(\F,k)$ be the collection of $k$-chains in $\F$ (where a $k$-chain, as before, is a set of $k$ points satisfying $\mathbf{a}_1\leq \ldots \leq \mathbf{a}_k$). Then $\C(\F,k)$ can be regarded as a subset of $\left([0,1]^n\right)^k$. This leads to the following natural question. By the \emph{measure} of a set $A\subset \mathbf{R}^n$ we always refer to the Lebesgue measure (or $n$-volume) of $A$ and denote it by $\lambda(A)$.

\begin{proble}\label{continuousproblem}
Given $n,M,k$, which measurable $A\subseteq [0,1]^n$ of measure $M$ minimizes the volume of $k$-chains, i.e.  $\lambda\left(\C(\F,k)\right)$?
\end{proble}

Consider the first non-trivial case, i.e. $n=k=2$. For $\mathbf{x}\in [0,1]^2$ define $M(\mathbf{x}):=\{\mathbf{y}\in A: \mathbf{x}\leq \mathbf{y}\}$. Let $S(A) := \{\mathbf{x}\in A : \nexists \mathbf{y}\in A : \mathbf{y}\leq \mathbf{x}\}$. Then it seems that in one of the optimal sets $A$ the function $f(x):=\lambda(M(x))$ should be constant on $\S$.
Giving a nice description of the optimal set $A$ in Problem~\ref{continuousproblem} may well turn out to be difficult. It may be possible to determine the limiting structure of the solution as $n,M$ remain fixed and $k$ grows to infinity. Alternatively, estimates on the minimal volume of $k$-chains might be of interest and easier to obtain. Let $f(n,M,k):=\inf\{\lambda(\C(A,k)): A\subseteq [0,1]^n,~ \lambda(A)=M\}$, where the infimum is taken over all measurable subsets $A$. 

\begin{proble}
Determine the value of $f\left(2,\frac12,2\right)$.
\end{proble}

\section{Acknowledgements}

We are very grateful to Jonathan Noel for pointing out that one of the open problems originally raised in this paper had already been considered by others.

\end{document}